\documentclass[reqno, 12pt]{amsart}

\usepackage{amsmath}
\usepackage{amsfonts}
\usepackage{amssymb}
\usepackage{mathrsfs}
\usepackage{amsthm}
\usepackage{mathtools}
\usepackage{bbm}
\usepackage[bbgreekl]{mathbbol}
\usepackage{graphicx, color}
\usepackage{tikz}
\usetikzlibrary{cd, intersections, calc, decorations.pathmorphing, arrows, decorations.pathreplacing}
\usepackage[all]{xy}

\DeclareSymbolFontAlphabet{\mathbb}{AMSb}
\DeclareSymbolFontAlphabet{\mathbbl}{bbold}

\usepackage{version}

\usepackage{stackengine}
\usepackage{calc}
\newlength\shlength
\newcommand\xshlongvec[2][0]{\setlength\shlength{#1pt}%
  \stackengine{-5pt}{$#2$}{\smash{$\kern\shlength%
    \stackengine{7.1pt}{$\mathchar"017E$}%
      {\rule{\widthof{$#2$}}{.57pt}\kern.4pt}{O}{r}{F}{F}{L}\kern-\shlength$}}%
      {O}{c}{F}{T}{S}}

\definecolor{refkey}{rgb}{0.9451,0.2706,0.4941}
\definecolor{labelkey}{rgb}{0.9451,0.2706,0.4941}

\usepackage{subfiles}
\usepackage[colorlinks, linkcolor=blue, citecolor=red, urlcolor=cyan,pagebackref]{hyperref} 

\numberwithin{equation}{section}

\usepackage{thmtools}
\usepackage{cleveref}

\declaretheorem[name=Theorem,numberwithin=section]{thm}
\declaretheorem[sharenumber=thm,name=Corollary]{cor}
\declaretheorem[sharenumber=thm,name=Lemma]{lem}
\declaretheorem[sharenumber=thm,name=Proposition]{prop}

\declaretheorem[name=Theorem]{introthm}

\theoremstyle{definition}

\declaretheorem[sharenumber=thm,name=Definition]{dfn}
\declaretheorem[sharenumber=thm,name=Example]{ex}

\declaretheorem[sharenumber=thm,name=Remark]{rem}
\declaretheorem[sharenumber=thm,name=Notation]{conv}

\crefname{thm}{Theorem}{Theorems}
\crefname{cor}{Corollary}{Corollaries}
\crefname{lem}{Lemma}{Lemmas}
\crefname{prop}{Proposition}{Propositions}

\crefname{introthm}{Theorem}{Theorems}
\crefname{introconj}{Conjecture}{Conjectures}
\crefname{introcor}{Corollary}{Corollaries}

\crefname{dfn}{Definition}{Definitions}
\crefname{ex}{Example}{Examples}
\crefname{claim}{Claim}{Claims}
\crefname{conj}{Conjecture}{Conjectures}
\crefname{rem}{Remark}{Remarks}
\crefname{conv}{Notation}{Notations}
\crefname{assum}{Assumption}{Assumptions}
\crefname{lemdef}{Lemma-Definition}{Lemma-Definitions}

\crefname{figure}{Figure}{Figures}
\crefname{section}{Section}{Sections}
\crefname{subsection}{Section}{Sections}
\crefname{appendix}{Appendix}{Appendices}

\newcommand{\bZ}{\mathbb{Z}}
\newcommand{\bQ}{\mathbb{Q}}
\newcommand{\bR}{\mathbb{R}}

\newcommand{\bN}{\mathbb{N}}

\newcommand{\A}{\mathcal{A}}

\newcommand{\cF}{\mathcal{F}}

\newcommand{\cS}{\mathcal{S}}

\newcommand{\cV}{\mathcal{V}}

\newcommand{\X}{\mathcal{X}}

\newcommand{\tri}{\triangle}
\newcommand{\sgn}{\mathrm{sgn}}

\newcommand{\lc}[2]{#1|_{#2,{\log}}}
\newcommand{\inprod}[2]{\langle #1, #2 \rangle}

\newcommand{\bs}{{\boldsymbol{s}}}
\newcommand{\bi}{{\mathbf{i}}}

\newcommand{\CV}{\mathsf{CV}}
\newcommand{\CA}{\mathscr{A}}
\newcommand{\UCA}{\mathscr{U}}

\newcommand{\pos}{\mathbb{R}_{>0}}
\newcommand{\trop}{\mathbb{R}^{\mathsf{T}}}

\newcommand{\bExch}{\mathsf{Exch}}
\newcommand{\uf}{\mathrm{uf}}

\newcommand{\Teich}{Teichm\"uller}

\newcommand{\ve}{\varepsilon}

\makeatletter
\newcommand{\oset}[3][0ex]{%
  \mathrel{\mathop{#3}\limits^{
    \vbox to#1{\kern-2\ex@
    \hbox{$\scriptstyle#2$}\vss}}}}
\makeatother
\newcommand{\overbar}[1]{\oset{#1}{-\!\!\!-\!\!\!-}}

\tikzset{
  qarrow/.style={->,shorten >=2pt,shorten <=2pt,>=latex
      },
}
\tikzset{
  head/.style={shorten >=4pt
      },
}
\tikzset{
  tail/.style={shorten <=4pt
      },
}
\newcommand{\qarrow}[2]{\draw[qarrow](#1) -- (#2)}

\tikzset{
  mid arrow/.style={postaction={decorate,decoration={
        markings,
        mark=at position .5 with {\arrow[#1]{stealth}}
      }}},
}

\tikzset{pics/.cd,
handle/.style={code={
\draw (-0.72,0) to[bend left] (0.72,0);
\draw (-0.9,0.1) to[bend right] (0.9,0.1);
}}}

\title[Fixed point theorem for cluster modular groups]
{Fixed point theorem for cluster modular groups}

\author{Tsukasa Ishibashi}
\address{Tsukasa Ishibashi, Mathematical Institute, Tohoku University, 
6-3 Aoba, Aramaki, Aoba-ku, Sendai, Miyagi 980-8578, Japan.}
\email{tsukasa.ishibashi.a6@tohoku.ac.jp}
\urladdr{https://sites.google.com/view/tsukasa-ishibashi/home} 

\date{\today}

\begin{document}

\begin{abstract}
We prove that any finite subgroup $G \subset \Gamma_{\boldsymbol{s}}$ of the cluster modular group has fixed points in the cluster manifolds $\mathcal{A}_{\boldsymbol{s}}(\mathbb{R}_{>0})$ and $\mathcal{X}_{\boldsymbol{s}}(\mathbb{R}_{>0})$ under a certain condition. This generalizes Kerckhoff's Nielsen realization theorem \cite{Kerckhoff} for the mapping class group action on the Teichm\"uller space. The condition holds whenever $\Gamma_{\boldsymbol{s}}$ admits a cluster DT transformation, and it can be also verified for all finite mutation types except for $X_7$. 
Our proof closely follows Kerckhoff's argument, based on the convexity of log-cluster variables.

\end{abstract}

\maketitle

\tableofcontents

\section{Introduction}
We say that a group $G$ acting on a set $X$ \emph{has a fixed point} if there is a point $x \in X$ satisfying $\gamma(x)=x$ for all $\gamma \in G$. 

For a compact oriented surface $\Sigma$ with negative Euler characteristic, the \emph{Nielsen realization problem}, named after J.~Nielsen, asks whether every finite subgroup $G$ of the mapping class group of $\Sigma$ has a fixed point in the \Teich\ space of $\Sigma$. Nielsen \cite{Nielsen} solved the problem for cyclic subgroups, equivalently for individual finite order elements.
The problem is solved affirmatively by Kerckhoff \cite{Kerckhoff}, based on the convexity of length functions along earthquake paths. 

In this paper, we aim to generalize this ``Nielsen realization theorem'' to \emph{cluster modular groups} \cite{FG09} (also known as \emph{cluster mapping class groups}). 

The cluster modular group $\Gamma_\bs$ is the automorphism group of the cluster structure -- cluster algebra \cite{FZ-CA1} or cluster variety \cite{FG09} -- defined by a mutation class $\bs$ of seeds. In particular, it acts on the \emph{cluster manifolds} $\A_\bs(\pos)$ and $\X_\bs(\pos)$, which are contractible real-analytic manifolds. There is a $\Gamma_\bs$-equivariant map $p: \A_\bs(\pos) \to \X_\bs(\pos)$, which we call the \emph{ensemble map}. 
In general, $p$ is a fiber bundle over its (possibly proper) image. 

For any marked surface $\Sigma$, we have an associated mutation class $\bs_\Sigma$ \cite{FST,FG07}, and the corresponding cluster-theoretic objects are related to \Teich\ theory as summarized in Table \ref{tab:cluster_Teich}. In this case, $p(\A_{\bs_\Sigma}(\pos))$ is identified with the usual \Teich\ space on which Kerckhoff worked. 

\begin{table}[ht]
    \centering
    \begin{tabular}{c|c}
    Cluster theory & \Teich\ theory \\ \hline
    $\A_{\bs_\Sigma}(\pos)$     & The decorated \Teich\ space \cite{Penner}  \\
    $\X_{\bs_\Sigma}(\pos)$     & The enhanced \Teich\ space \cite{Fock,FG07}  \\
    $\Gamma_{\bs_\Sigma}$   & The tagged mapping class group \cite{FST,BS15}
    \end{tabular}
    \label{tab:cluster_Teich}
\end{table}

We are going to consider if any finite subgroup $G \subset \Gamma_\bs$ has a fixed point in $\A_\bs(\pos)$ or $\X_\bs(\pos)$, for a general mutation class $\bs$. 

\subsection{Nielsen realization theorem}
The following is our main theorem:

\begin{introthm}[\cref{thm:fixed_point_filling}]\label{introthm:fixed_point_filling}
Assume that there exists a filling set $\Lambda \subset \UCA^+_\bs$ (\cref{def:filling_set}). 
Then, any finite subgroup $G \subset \Gamma_\bs$ has a fixed point in $\A_\bs(\pos)$. In particular, it also has a fixed point in $p(\A_\bs(\pos)) \subset \X_\bs(\pos)$.
\end{introthm}
This result gives a cluster-theoretic generalization of the Nielsen realization theorem due to Kerckhoff. It also generalizes \cite[Proposition 2.3 (ii) $\Longrightarrow$ (iii)]{Ish19}, which was stated for finite order elements of $\Gamma_\bs$. See also \cite{dSG} and \cite{HLY} for related results for finite order elements, obtained by different techniques.

The filling set $\Lambda$ plays the role of a collection of simple closed geodesics that fills up the surface. Our proof of \cref{introthm:fixed_point_filling} closely follows that of Kerckhoff \cite{Kerckhoff}: the $G$-orbit of $\Lambda$ produces a real-analytic $G$-invariant function $L_G$ on $\A_\bs(\pos)$ with certain convexity properties. We prove that it has a unique minimizer in $\A_\bs(\pos)$, which gives a $G$-fixed point. Its image under $p$ gives a $G$-fixed point in $\X_\bs(\pos)$ as well. 

The remaining problem is how to find a filling set. 
The following gives vast class of examples to which we can apply \cref{introthm:fixed_point_filling}:

\begin{introthm}[\cref{cor:fixed_point_DT}]
Assume that $\Gamma_\bs$ admits a cluster Donaldson--Thomas transformation. Then, any finite subgroup $G \subset \Gamma_\bs$ has fixed poins in $\A_\bs(\pos)$ and $\X_\bs(\pos)$.
\end{introthm}
Here, the \emph{cluster Donaldson--Thomas transformation} (\emph{cluster DT transformation} for short) is a special element of $\Gamma_\bs$, which is unique if exists. The mutation classes $\bs$ admitting a cluster DT transformation include:
\begin{enumerate}
    \item All the rank $\leq 2$ cases,
    \item Those of acyclic quivers \cite{BDP14}, 
    \item Those of $T_{\mathbf{n},\mathbf{w}}$-quivers studied in \cite{KG24}, which include all finite types, affine types, and doubly-extended Dynkin types. 
    \item Those associated with the moduli spaces of $G$-local systems on $\Sigma$ \cite{GS18,GS19}, except for the cases where $G=SL_2$ and $\Sigma$ is a once-punctured closed surface. 
    \item Those associated with the $K$-theoretic Coulomb branches of quiver gauge theories \cite{SS} whose gauge quivers have no loops,
    \item Those associated with any braid varieties \cite{CGGLSS}, 
\end{enumerate}
among others. See \cite{KellerDemonet} for a survey. 
Therefore, we have a generalization of the Nielsen realization theorem for these classes. 
For example, in the case (4), the manifold $\X_{\bs}(\pos)$ is the Fock--Goncharov higher \Teich\ space \cite{FG06}, and thus \cref{introthm:fixed_point_filling} gives the higher-rank analogue of Kerckhoff's Nielsen realization theorem. Even in the $SL_2$-case, our proof provides an independent proof of Kerckhoff's result, based on the convexity of log-cluster variables (logarithms of $\lambda$-length functions) rather than that of length functions along closed geodesics. 

We can also find filling sets for the mutation classes associated with once-punctured closed surfaces, which do not admit a cluster DT transformation (\cref{prop:once-punctured}). Together with the existence results of cluster DT transformations in the other cases, we obtain:

\begin{introthm}[\cref{thm:fixed_point_finite_mutation}]
Let $\bs$ be any mutation class of finite mutation type, except for type $X_7$. Then, any finite subgroup $G \subset \Gamma_\bs$ has fixed points in $\A_\bs(\pos)$ and $\X_\bs(\pos)$. 
\end{introthm}
For the remaining case $X_7$, see \cref{rem:X_7}.

\subsection{Nielsen realization theorem for cluster realizations of Weyl groups}
We provide another approach to the Nielsen realization problem for the Weyl group actions constructed in \cite{IIO}, which may be of independent interest.
For any skew-symmetrizable Kac--Moody Lie algebra $\mathfrak{g}$ and an integer $m \geq 2$, one can construct a quiver $Q_m(\mathfrak{g})$ such that 
we have a group embedding 
\begin{align*}
    \phi_m: W(\mathfrak{g}) \to \Gamma_{\bs_m(\mathfrak{g})}
\end{align*}
of the Weyl group $W(\mathfrak{g})$ into the cluster modular group associated with the 
mutation class $\bs_m(\mathfrak{g})$ containing $Q_m(\mathfrak{g})$. In this case, the $W(\mathfrak{g})$-action is trivial on $p(\A_{\bs_m(\mathfrak{g})})$, and it's ``purely'' along the fiber direction of the ensemble map (\cref{lem:action_peripheral}). In particular, the Nielsen realization theorem for $\X_{\bs_m(\mathfrak{g})}$ holds obviously.

These behaviors stand in sharp contrast to the original \Teich\ case. 
We prove the Nielsen realization theorem for $\A_{\bs_m(\mathfrak{g})}(\pos)$ by a completely different argument:

\begin{introthm}[\cref{thm:Weyl_fixed_point}]
Let $\mathfrak{g}$ be a skew-symmetrizable Kac--Moody Lie algebra, and $m \geq 2$ an integer. Then, for any finite subgroup $G \subset W(\mathfrak{g})$, $\phi_m(G)$ has a fixed point in $\A_{\bs_m(\mathfrak{g})}(\pos)$. 
\end{introthm}
We prove this theorem by constructing a $W(\mathfrak{g})$-equivariant section in the case of finite type (\cref{lem:Weyl_finite}), and then deduce the general case using Tits' theorem, which states that any finite subgroup of $W(\mathfrak{g})$ is contained in a spherical parabolic subgroup. We also note that the affine type $A$ case provides examples admitting no $W(\mathfrak{g})$-equivariant section of the ensemble map (\cref{rem:Weyl_section} (2)).

\subsection*{Organization of the paper}
In \cref{app:convex}, we review basic facts about convex functions that will be used throughout the paper. \cref{app:cluster} summarizes basic definitions and necessary results from cluster algebra. Our Nielsen realization theorems are proved in \cref{sec:fixed_point_general}. Additional details for the fixed point theorem for finite-order elements in \cite{Ish19} is provided in Appendix \ref{app:periodic}, which are independent of the body text. 

\subsection*{Acknowledgements}
The author is partially supported by JSPS KAKENHI Grant Number JP24K16914. He is grateful to Nariya Kawazumi for inspiring him to work on this problem and for valuable comments on the draft of this paper.

\section{Convex functions}\label{app:convex}
We review basic facts about convex functions. A concise account is found in \cite{FundamentalConvex}.

Recall that a function $f: \Omega \to \bR$ defined on a convex domain $\Omega \subset \bR^N$ is said to be \emph{convex} if it satisfies
\begin{align}\label{eq:convex}
    f(\lambda x + (1-\lambda) y) \leq \lambda f(x) + (1-\lambda) f(y)
\end{align}
for all $x,y \in \Omega$ and $0\leq \lambda \leq 1$. It is said to be \emph{strictly convex} if the strict inequality in \eqref{eq:convex} holds for all $x,y \in \Omega$ with $x \neq y$ and $0< \lambda < 1$. 

It is clear from the definition that $f:\Omega \to \bR$ is convex if and only if $f|_I: I \to \bR$ is convex for any embedded interval $I \subset \Omega$. 

\begin{lem}
A convex function $f:  (a,b) \to \bR$ satisfies
\begin{align}\label{eq:convex_inequality}
    \frac{f(t)-f(s)}{t-s} \leq \frac{f(u)-f(s)}{u-s} \leq \frac{f(u)-f(t)}{u-t}
\end{align}
for all $s < t < u$ in $(a,b)$.
\end{lem}

\begin{proof}
Write $t=\lambda s + (1-\lambda) u$ for some $0<\lambda<1$. We have $\lambda = (u-t)/(u-s)$. Then by \eqref{eq:convex}, we get
\begin{align*}
    f(t) &\leq \frac{u-t}{u-s} f(s) + \bigg(1- \frac{u-t}{u-s}\bigg) f(u) \\ 
    &=  \frac{u-t}{u-s} (f(s)-f(u)) +  f(u).
\end{align*}
Therefore we get 
\begin{align*}
    \frac{f(u)-f(s)}{u-s} \leq \frac{f(u)-f(t)}{u-t}.
\end{align*}
Similarly, by writing $t = (1-\lambda)s+ \lambda u$ with $\lambda=(t-s)/(u-s)$, we get
\begin{align*}
    \frac{f(t)-f(s)}{t-s} \leq \frac{f(u)-f(s)}{u-s}.
\end{align*}
\end{proof}

\begin{cor}\label{cor:convex_continuous}
A convex function is continuous and directionally differentiable.
\end{cor}

\begin{proof}
The previous lemma tells us that both the limits 
\begin{align*}
    &\left.\frac{d}{dt}\right|_{t=s+0} f(t)= \lim_{t \to s+0} \frac{f(t)-f(s)}{t-s}, \\ 
    &\left.\frac{d}{dt}\right|_{t=u-0} f(t)=\lim_{t \to u-0} \frac{f(u)-f(t)}{u-t}
\end{align*}
exist for any $s,u \in (a,b)$. In particular, $\lim_{t\to s+0} f(t)=f(s)$ and $\lim_{t\to u-0} f(t)=f(u)$. 
Therefore $f(x)$ is continuous and directionally differentiable at any $x \in (a,b)$.
\end{proof}
For a convex function $f: \bR^n \to \bR$ and $x \in \bR^n$, denote the directional derivative by
\begin{align*}
    Df_x(v) := \lim_{t \to +0} \frac{f(x+tv)-f(x)}{t}
\end{align*}
for $v \in \bR^n$. 

\subsection{Midpoint convexity}

A function $f: (a,b) \to \bR$ is said to be \emph{midpoint-convex} if it satisfies
\begin{align*}
    f\bigg(\frac{x+y}{2}\bigg) \leq \frac{f(x)+f(y)}{2}
\end{align*}
for any $x,y \in (a,b)$. It is \emph{strictly midpoint-convex} if the strict inequality holds for all $x,y \in (a,b)$ with $x \neq y$.

\begin{lem}[rational convexity]\label{prop:convex_rational}
A midpoint-convex function $f: (a,b) \to \bR$ satisfies
\eqref{eq:convex} for all $x,y \in (a,b)$ and $\lambda \in [0,1]\cap\bQ$. 
If $f$ is strictly midpoint-convex, then the strict inequality holds in \eqref{eq:convex} for $x\neq y$ and $0 < \lambda <1$. 
\end{lem}

\begin{proof}
We claim
\begin{align}\label{eq:convex_rational}
    f \bigg( \frac{x_1+\dots+x_k}{k} \bigg) \leq \frac{f(x_1)+\cdots+f(x_k)}{k}
\end{align}
for any integer $k \geq 1$ and $x_1,\dots,x_k \in (a,b)$. When $k=2^n$ for some $n$, it follows by an induction on $n$. When $2^{n-1}<k \leq 2^n$ for some $n$, let $\overline{x}:=(x_1+\dots+x_k)/k \in (a,b)$. Then
\begin{align*}
    f(\overline{x}) = f \bigg( \frac{x_1+\dots+x_k + (2^n-k)\overline{x}}{2^n} \bigg) \leq \frac{f(x_1)+\dots+f(x_k) + (2^n-k)f(\overline{x})}{2^n},
\end{align*}
where we used the established claim for $2^n$. It implies $f(\overline{x}) \leq (f(x_1)+\dots+f(x_k))/k$. Therefore the claim is proved. 

For any $\lambda=p/q \in [0,1]\cap\bQ$, we apply \eqref{eq:convex_rational} for $k=q$ and get
\begin{align*}
    f(\lambda x + (1-\lambda)y) = f \bigg(\frac{px+(q-p)y}{q} \bigg) \leq \frac{pf(x) + (q-p)f(y)}{q} = \lambda f(x) + (1-\lambda)f(y)
\end{align*}
for any $x,y \in (a,b)$.

If $f$ is strictly midpoint-convex, then each inequality in the argument above becomes strict. 
\end{proof}

\begin{prop}\label{prop:midpoint_convex_extension}
A continuous midpoint-convex function $f: (a,b) \to \bR$ is convex. 
If moreover $f$ is strictly midpoint-convex, then $f$ is strictly convex. 
\end{prop}

\begin{proof}
The first statement follows immediately from \cref{prop:convex_rational} and continuity. 
Suppose that $f$ is strictly midpoint-convex. If the convexity of $f$ were not strict, then there exists $x,y \in (a,b)$, say $x<y$, such that \eqref{eq:convex} becomes an equality for some $0<\lambda<1$. Let $z:=\lambda x + (1-\lambda) y$. Then $f$ is constant either on $[x,z]$ or $[z,y]$ by convexity, which contradicts to the strict inequality in \cref{prop:convex_rational}.
\end{proof}

\subsection{Convexity of Log-Laurent functions}
We are going to consider functions of the form 
\begin{align}\label{eq:log-poly}
    f(x_1,\dots,x_n):= \log F(e^{x_1},\dots,e^{x_n}),
\end{align}
where $F \in \bZ_+[X_1^{\pm 1},\dots,X_n^{\pm 1}]$ is a Laurent polynomial with positive integral coefficients. We will use the multi-index notations
\begin{align}\label{eq:multi_index}
    F(X)=\sum_{\alpha \in \bZ^n} c_\alpha X^\alpha, \quad f(x) = \log F(e^x) = \log \bigg( \sum_{\alpha \in \bZ^n} c_\alpha e^{\inprod{\alpha}{x}} \bigg),
\end{align}
where $X^\alpha:=\prod_{j=1}^n X_j^{\alpha_j}$, $e^x:=(e^{x_1},\dots,e^{x_n})$, $\inprod{\alpha}{x}:= \sum_{j=1}^n \alpha_j x_j$ for $X=(X_1,\dots,X_n)$, $x=(x_1,\dots,x_n)$ and $\alpha=(\alpha_1,\dots,\alpha_n)$. 

\begin{lem}\label{lem:log-poly_convex}
The function \eqref{eq:log-poly} is convex on $\bR^n$.
\end{lem}

\begin{proof}
Since $f(x_1,\dots,x_n)$ is clearly continuous, it suffices to prove the mid-convexity by \cref{prop:midpoint_convex_extension}. The mid-convexity condition is equivalent to
\begin{align*}
    F(e^{(x+y)/2})^2 \leq F(e^x) \cdot F(e^y). 
\end{align*}
From \cref{eq:multi_index}, we have
\begin{align*}
    F(e^{(x+y)/2}) = \sum_{\alpha \in \bZ^n} c_\alpha e^{\inprod{\alpha}{x+y}/2}.
\end{align*}
By applying the Cauchy--Schwarz inequality
\begin{align*}
    \bigg(\sum_{\alpha \in \bZ^n} u_\alpha v_\alpha \bigg)^2 \leq \bigg(\sum_{\alpha \in \bZ^n} u_\alpha^2 \bigg) \cdot \bigg(\sum_{\alpha \in \bZ^n} v_\alpha^2 \bigg)
\end{align*}
on the normed space $L^2(\bZ^n)$ to $u_\alpha:=c_\alpha^{1/2}e^{\inprod{\alpha}{x}/2}$ and $v_\alpha:=c_\alpha^{1/2}e^{\inprod{\alpha}{y}/2}$, we get the desired inequality. 
\end{proof}
By the Cauchy--Schwarz proof, the equality in the convexity condition holds if and only if $e^{\inprod{\alpha}{x}/2} = c e^{\inprod{\alpha}{y}/2}$ for a constant $c \in \bR$. 
Based on this observation, let us investigate the condition for strict convexity.

\begin{dfn}
Given a Laurent polynomial $F \in \bZ[X_1^{\pm 1},\dots,X_n^{\pm 1}]$, its \emph{support} is defined to be
\begin{align*}
    \mathrm{Supp}(F):= \{ \alpha \in \bZ^n \mid c_\alpha \neq 0\}
\end{align*}
when we write $F=\sum_{\alpha \in \bZ^n} c_\alpha X^\alpha$. Define 
\begin{align*}
    \mathrm{Slope}(F):= \mathrm{span}_\bR\{ \alpha - \beta \mid \alpha,\beta \in \mathrm{Supp}(F)\}.
\end{align*}
\end{dfn}

\begin{prop}\label{lem:strict_convex}
The function \eqref{eq:log-poly} is strictly convex if $\mathrm{Slope}(F)=\bR^n$. More precisely, it is strictly convex along a straight line with directional vector $v \in \bR^n$ if and only if $v \notin \mathrm{Slope}(F)^\perp$.
\end{prop}

\begin{proof}
It suffices to prove the second statement. Fix a directional vector $v \in \bR^n$. Fixing an initial point $x \in \bR^m$, consider the functions
\begin{align*}
    F_{v;x}(t):=F(e^{x+tv}), \quad f_{v;x}(t):=\log F_{v;x}(t)
\end{align*}
for $t \in \bR$. 
Then we get the midpoint-convexity
\begin{align}\label{eq:convex_direcion}
    F_{v;x}\bigg(\frac{t+s}{2}\bigg)^2 \leq F_{v;x}(t)\cdot F_{v;x}(s)
\end{align}
for $t,s \in \bR$ by a Cauchy-Schwarz argument similar to \cref{lem:log-poly_convex}. In particular, the equality holds if and only if there exists $c>0$ such that $e^{t\inprod{\alpha}{v}}=c e^{s\inprod{\alpha}{v}}$ for all $\alpha \in \mathrm{Supp}(F)$. It implies  $(t-s)\inprod{\alpha}{v}= (t-s) \inprod{\beta}{v}$ for all $\alpha-\beta \in \mathrm{Slope}(F)$. 

If $v \notin \mathrm{Slope}(F)^\bot$, then there exists at least one $\alpha -\beta \in \mathrm{Slope}(F)$ such that $\inprod{\alpha-\beta}{v} \neq 0$. Therefore we see that the equality in \eqref{eq:convex_direcion} holds only when $t=s$, hence $f_{v;x}: \bR \to \bR$ is strictly midpoint-convex. By continuity and \cref{prop:midpoint_convex_extension}, $f_{v;x}$ is strictly convex.

On the other hand, if $v \in \mathrm{Slope}(F)^\bot$, then by choosing $\beta \in \mathrm{Supp}(F)$, we can write
\begin{align*}
    F_{v;x}(t) &= e^{\inprod{\beta}{x+tv}}\sum_{\alpha \in \mathrm{Supp}(F)} c_\alpha e^{\inprod{\alpha-\beta}{x+tv}} \\
    &= e^{\inprod{\beta}{x+tv}}\sum_{\alpha \in \mathrm{Supp}(F)} c_\alpha e^{\inprod{\alpha-\beta}{x}} = e^{t\inprod{\beta}{v}} F_{v;x}(0).
\end{align*}
Then $f_{v;x}(t)=\log F_{v;x}(t)$ is a linear function of $t$, which is not strictly convex.
\end{proof}

\subsection{Minima}
Recall that the convex hull $\mathrm{Conv}(S)$ of any subset $S \subset \bR^n$ is the smallest convex subset containing $S$. 
The convex hull of finitely many vectors $\alpha_1,\dots,\alpha_r \in \bR^n$ is given by
\begin{align*}
    \mathrm{Conv}(\alpha_1,\dots,\alpha_r)=\bigg\{\sum_{k=1}^r m_j \alpha_j ~\bigg| ~ m_j \geq 0,\ \sum_{k=1}^r m_j=1 \bigg\}.
\end{align*}
See 
\cite[Example A.1.3.5]{FundamentalConvex}. 
The \emph{recession cone} (or \emph{asymptotic cone}) of a non-empty closed convex subset $C \subset \bR^n$ is defined to be
\begin{align*}
    C_\infty:= \{ v \in \bR^n \mid \text{$x+tv \in C$ for all $t>0$}\},
\end{align*}
which does not depend on $x \in C$ \cite[Proposition A.2.2.1]{FundamentalConvex}. 
A convex subset $C \subset \bR^n$ is said to be \emph{full-dimensional} if its affine closure is $\bR^n$. 

\begin{prop}\label{lem:proper_convex}
Let $F_1,\dots,F_r \in \bZ_+[X_1^{\pm 1},\dots,X_n^{\pm 1}]$ be Laurent polynomials with positive integral coefficients, and let $S \subset \bigcup_{k=1}^r \mathrm{Supp}(F_k)$ be a subset such that $\mathrm{Conv}(S)$ is full-dimensional, and its interior contains $0$.
Then the sublevel sets of the function
\begin{align*}
    f(x):=\max_{k=1,\dots,r} f_k(x)
\end{align*}
are compact. 
Here, each $f_k$ is the function associated to $F_k$ as in \eqref{eq:log-poly}. In particular, $f$ attains a minimum in $\bR^n$. 
\end{prop}

\begin{proof}
For $R>0$, we claim that the sublevel set $K_R:=f^{-1}((-\infty,R]) \subset \bR^n$ is compact. Firstly, it is closed since $f$ is continuous as the maximum of finitely many continuous functions. 
For $x \in K_R$, the inequality $f(x) \leq R$ implies that $f_k(x) \leq R$ for all $k=1,\dots,r$. For any $\alpha \in \mathrm{Supp}(F_k)$,
we may write
\begin{align*}
    f_k(x) = \log \bigg( c_{\alpha} e^{\inprod{\alpha}{x}} + \sum_{\beta \neq \alpha} c_\beta e^{\inprod{\beta}{x}}\bigg) = \inprod{\alpha}{x} + \log\bigg( c_{\alpha_k} + \sum_{\beta \neq \alpha} c_\beta e^{\inprod{\beta-\alpha}{x}}\bigg).
\end{align*}
Then we get $\inprod{\alpha}{x} \leq R$, since all the coefficients are positive integers. 
Hence the sublevel set is contained in the closed convex set 
\begin{align*}
    C:= \{ x \in \bR^n \mid \text{$\inprod{\alpha}{x} \leq R$ for all $\alpha \in S$}\}.
\end{align*}
It suffices to prove that $C$ is bounded. 
Its recession cone is given by
\begin{align*}
    C_\infty=\{ v \in \bR^n \mid \text{$\inprod{\alpha}{v} \leq 0$ for all $\alpha \in S$}\}.
\end{align*}
Since the interior of $\mathrm{Conv}(S)$ contains $0$, there exists $m_\alpha > 0$ for all $\alpha \in S$ satisfying $\sum_{\alpha \in S} m_\alpha \alpha =0$ and $\sum_{\alpha \in S}m_\alpha =1$. Then any $v \in C_\infty$ satisfies $0 = \sum_{\alpha \in S} m_\alpha \inprod{\alpha}{v} \leq 0$, which implies $\inprod{\alpha}{v}=0$ for all $\alpha \in S$. Since $S$ linearly spans $\bR^n$, it follows that $v=0$. Therefore $C_\infty=\{0\}$, and hence $C$ is bounded by \cite[Proposition A.2.2.3]{FundamentalConvex}.

Therefore $K_R \subset C$ is compact.
It follows immediately that $f$ attains a minimum.
\end{proof}

\begin{prop}[{\cite[Theorem D.2.2.1]{FundamentalConvex}}]\label{prop:min_subgrad}
Let $f: \bR^n \to \bR$ be a convex function. Then $x \in \bR^n$ is a minimizer of $f$ if and only if $0 \in \partial f(x)$. Here,
\begin{align*}
    \partial f(x):= \{s \in \bR^n \mid \text{$\inprod{s}{v} \leq D_x f(v)$ for all $v \in \bR^n$} \}
\end{align*}
denotes the set of subgradients.
\end{prop}

\section{Cluster algebra}\label{app:cluster}
We closely follow the definitions of \cite[Chapter II]{Nak_CA}, while using Fock--Goncharov notation and convention \cite{FG09}.\footnote{Note that our exchange matrices, $C$- and $G$-matrices are transpose of those in \cite{Nak_CA}.} 
Fix a finite index set $I$ and a tuple $d=(d_i)_{i \in I}$ of positive integers. 
Let $\cF_A,\cF_X$ be two fields each isomorphic to the field $\bQ(u_i \mid i \in I)$ of rational functions in $|I|$ variables with coefficients in $\bQ$. 
\subsection{Seeds and mutations}
A (labeled, skew-symmetrizable) seed $\bi = (\ve,\mathbf{A},\mathbf{X})$ in $(\cF_A,\cF_X)$ consists of the following data:
\begin{itemize}
    \item $\ve=(\ve_{ij})_{i,j \in I}$ is an $I \times I$ integral matrix such that $(\ve_{ij}d_j)_{i,j \in I}$ is skew-symmetric. 
    \item $\mathbf{A} = (A_i)_{i \in I}$ is a tuple of elements of $\cF_A$ such that $\cF_A \cong \bQ(A_i \mid i \in I)$. 
    \item $\mathbf{X} = (X_i)_{i \in I}$ is a tuple of elements of $\cF_X$ such that $\cF_X \cong \bQ(X_i \mid i \in I)$. 
\end{itemize}
We call $\ve$ an exchange matrix, and each $A_i$ (resp. $X_i$) a cluster $K_2$-(resp. Poisson) variable. 
For $k \in I$, the seed mutation of $\bi$ in the direction $k$ produces a new seed $\mu_k(\bi) = (\ve',\mathbf{A}',\mathbf{X}')$ as follows:

\begin{align}
\ve'_{ij}&=
\begin{cases}
    -\ve_{i j} & \mbox{if $i=k$ or $j=k$,}\\
    \ve_{i j}+[\ve_{i k}]_+ [\ve_{k j}]_+ -[-\ve_{i k}]_+[-\ve_{k j}]_+ & \mbox{otherwise}.
\end{cases} \label{eq:mutation}\\
A_i'&= 
\begin{cases} 
    A_k^{-1}\left(\prod_{j\in I} A_{j}^{[\varepsilon_{kj}]_+}+ \prod_{j \in I} A_{j}^{[-\varepsilon_{kj}]_+}\right) & \mbox{$i=k$}, \\
    A_i & \mbox{if $i\neq k$},
\end{cases} \label{eq:A-transf}\\
 X'_i&:= 
\begin{cases}
    X_k^{-1} & i=k,\\
    X_i\,(1 + X_k^{-\mathrm{sgn}(\ve_{ik})})^{-\ve_{ik}} & i \neq k.
\end{cases} \label{eq:X-transf}
\end{align}
Here, $[x]_+:=\max\{x,0\}$ for $x \in \bR$, and $\sgn(x):=x/|x|$ for $x \neq 0$.


We say that two seeds $\bi$ and $\bi'$ in $\cF$ are mutation-equivalent to each other if $\bi'$ can be obtained from $\bi$ by a sequence of seed mutations and simultaneous change of labels by permutations on $I$.
An equivalence class of seeds is called a mutation class, and denoted by $\bs$. 

Given a seed $\bi \in \bs$, we denote the associated data by $\ve^\bi=(\ve_{ij}^\bi)$, $\mathbf{A}^{\!\bi}=(A_i^\bi)_{i \in I}$, and $\mathbf{X}^{\bi}=(X_i^\bi)_{i \in I}$.
Let $\CV_\bs:=\bigcup_{\bi \in \bs} \mathbf{A}^{\!\bi}$ denote the set of all the cluster variables.

\begin{thm}[Laurent phenomenon \cite{FZ-CA1}]\label{thm:Laurent_phenomenon}
Fix any seed $\bi \in \bs$. Then, each cluster variable $A \in \CV_\bs$ can be written as a Laurent polynomial of $A_j^{\bi}$, $j \in I$ with integral coefficients.
\end{thm}

\begin{thm}[Laurent positivity \cite{GHKK}]\label{thm:Laurent_positive}
All the coefficients in the Laurent expression of $A \in \CV_\bs$ above are positive.
\end{thm}
The structure of Laurent expression can be described in the following way (see \cite{IK19} for our convention):

\begin{thm}[Separation formula \cite{FZ-CA4}]\label{thm:separation}
For any $A \in \CV_\bs$, its Laurent expression for $\bi$ is written as
\begin{align}\label{eq:separation_formula}
    A= \prod_{j \in I} (A_j^\bi)^{g(A)^\bi_j} \cdot F^\bi_A(p^\ast \mathbf{X}^\bi),
\end{align}
where 
\begin{itemize}
    \item $\mathbf{g}(A)^\bi=(g(A)^\bi_j) \in \bZ^I$, and $F^\bi_A \in \bZ[y_1,\dots,y_n]$ has constant term $1$.  
    \item $p^\ast X_k^\bi :=\prod_{i \in I} (A_j^\bi)^{\ve_{kj}^\bi}$ for all $k \in I$.
\end{itemize}
The data $\mathbf{g}(A)^\bi$ and $F_A^\bi$ are called the $g$-vector and the $F$-polynomial of $A$ with respect to the initial seed $\bi$, respectively.
\end{thm}


\subsubsection*{$C$- and $G$-matrices}
Let $\bs$ be a mutation class of seeds. Fix any initial seed $\bi_0 \in \bs$. Consider the extended exchange matrix
\begin{align*}
    \widetilde{\ve}^{\bi_0}:= (\ve^{\bi_0} \ |\ \mathrm{Id}),
\end{align*}
where $\mathrm{Id}$ denotes the $I\times I$ identity matrix. Then by successively applying the same mutation rule \eqref{eq:mutation} to $\widetilde{\ve}^{\bi_0}$, we get
\begin{align*}
    \widetilde{\ve}^{\bi}= (\ve^{\bi}\ |\ C^{\bi;\bi_0})
\end{align*}
for some integral matrix $C^{\bi;\bi_0}$ for each $\bi \in \bs$, which is called the \emph{$C$-matrix} with respect to the initial seed $\bi_0$. 

The integral matrix $G^{\bi;\bi_0}:=(g(A_i^\bi)_j^{\bi_0})_{i,j \in I}$ is called the \emph{$G$-matrix}. The following is useful:

\begin{thm}[Tropical duality {\cite[(3.11)]{NZ12}}]\label{thm:tropical_duality}
We have 
\begin{align*}
    G^{\bi;\bi_0} = D((C^{\bi;\bi_0})^\top)^{-1} D^{-1},
\end{align*}
where 
$D:=\mathrm{diag}(d_i \mid i \in I)$.
\end{thm}

\begin{rem}\label{rem:sync}
By the synchronicity theorem \cite{Nak}, all the periodicities of the mutation rules \eqref{eq:mutation}, \eqref{eq:A-transf}, \eqref{eq:X-transf} are the same. Namely, for any seed $\bi'$ obtained from $\bi$ by a sequence of mutations, the following are equivalent:
\begin{itemize}
    \item $\ve^{\bi'}=\ve^\bi$.
    \item $\mathbf{A}^{\!\bi'}=\mathbf{A}^{\!\bi}$.
    \item $\mathbf{X}^{\bi'}=\mathbf{X}^\bi$.
\end{itemize}
Moreover, we have $C^{\bi';\bi}=\mathrm{Id}$ if and only if $\bi'=\bi$. 
\end{rem}

\subsubsection*{Cluster algebras}

\begin{dfn}
Let $\bs$ be a mutation class of seeds.
\begin{enumerate}
    \item The subalgebra $\CA_\bs \subset \cF_A$ generated by $\CV_\bs$ is called the \emph{cluster algebra} associated with the mutation class $\bs$. 
    \item The \emph{upper cluster algebra} (also known as \emph{universally Laurent algebra}) is defined to be
\begin{align*}
    \UCA_\bs:= \bigcap_{\bi \in \bs} \bZ[(A_j^\bi)^{\pm 1} \mid j \in I] \subset \cF_A.
\end{align*}
    \item The \emph{universally positive Laurent semiring} is the sub-semiring $\UCA^+_\bs \subset \UCA_\bs$ consisting of elements whose Laurent expressions have positive coefficients for any $\bi \in \bs$. 
\end{enumerate}

By \cref{thm:Laurent_phenomenon} and \cref{thm:Laurent_positive}, we have inclusions $\CA_\bs \subset \UCA^+_\bs \subset \UCA_\bs$.
\end{dfn}

\subsection{Cluster manifolds}
We quickly introduce the cluster manifolds $\A_\bs(\pos)$ and $\X_\bs(\pos)$, which are originally defined as positive real parts of cluster varieties $\A_\bs$ and $\X_\bs$, respectively in \cite{FG09}. 

\begin{dfn}
Given a mutation class $\bs$ of seeds, the associated \emph{cluster $K_2$-manifold} is a real-analytic manifold $\A_\bs(\pos)$ whose atlas consisting of global charts
\begin{align*}
    \psi_\bi: \A_\bs(\pos) \xrightarrow{\sim} \bR_{>0}^I
\end{align*}
parameterized by $\bi \in \bs$
such that the transition function $\psi_{\bi'} \circ \psi_\bi^{-1}: \bR_{>0}^I \to \bR_{>0}^I$ is given by the mutation formula \eqref{eq:A-transf} for any mutation $\bi'=\mu_k(\bi)$. 
Similarly, we define the \emph{cluster Poisson manifold} $\X_\bs(\pos)$ by using the transformation rule \eqref{eq:X-transf} instead.
\end{dfn}
In particular, both $\A_\bs(\pos)$ and $\X_\bs(\pos)$ are homeomorphic to $\bR_{>0}^n$ via cluster charts. In this paper, we will mainly work in the log-chart
\begin{align}\label{eq:log_chart}
    \log \psi_\bi: \A_\bs(\pos) \to \bR_{>0}^I \xrightarrow{\log} \bR^I,
\end{align}
where the second map is defined by $(u_j)_{j \in I} \mapsto (\log u_j)_{j \in I}$.

The following is immediate from definitions:

\begin{lem}
Any element $F \in \UCA_\bs$ defines a real-analytic function $F: \A_\bs(\pos) \to \bR_{>0}$.
\end{lem}

\subsection{Cluster modular group}\label{subsec:cluster_modular}
The combinatorial relations among the seeds $\bi \in \bs$ are organized in the following way.
The \emph{labeled exchange graph} $\bExch_\bs$ is the graph with vertices given by the seeds $\bi \in \bs$ and labeled edges of the following two types:
\begin{itemize}
    \item edges of the form $\bi \overbar{k} \bi'$ corresponding to the mutation $\bi'=\mu_k(\bi)$ for $k \in I_\uf$;
    \item edges of the form $\bi \overbar{(i\ j)} \bi'$ corresponding to the transposition $(i\ j)$ of labels with $i,j \in I$.
\end{itemize}
Note that each vertex $\bi$ of $\bExch_\bs$ is colored by the exchange matrix $\ve^\bi$. 

\begin{dfn}[Cluster modular group]
The \emph{cluster modular group} $\Gamma_\bs \subset \mathrm{Aut}(\bExch_\bs)$ is the subgroup consisting of graph automorphisms that preserves the vertex coloring and the edge labeling.
\end{dfn}
In other words, $\Gamma_\bs$ is the permutation group of seeds in $\bs$ that preserves the exchange matrix and commutes with mutations and permutations. 

An element $\phi \in \Gamma_\bs$ corresponds to a ``mutation loop'' in \cite{KT15} in the following way. Fixing a vertex $\bi \in \bExch_\bs$, we can take an edge path in $\bExch_\bs$
\begin{align*}
    \gamma: \bi=\bi_0 \xrightarrow{\alpha_1} \bi_1 \xrightarrow{\alpha_2} \dots \xrightarrow{\alpha_L} \bi_L=\phi^{-1}(\bi).
\end{align*}
Here, $\alpha_j$ denotes the label on the edge between $\bi_{j-1}$ and $\bi_j$, corresponding either to a mutation or a transposition of labels. We have the boundary condition $\ve^{\bi_0}=\ve^{\bi_L}$, and hence $\gamma$ is a mutation loop in the sense of \cite{KT15}. 
We call $\gamma$ a \emph{representation path} of $\phi$.

The cluster modular group $\Gamma_\bs$ naturally acts on $\A_\bs(\pos)$ so that
\begin{align}\label{eq:def_action}
    &A_j^\bi(\phi(g)) = A_j^{\phi^{-1}(\bi)}(g) 
\end{align}
for $g \in \A_\bs(\pos)$, $\phi \in \Gamma_\bs$, $\bi \in \bs$ and $j \in I$. It acts on $\X_\bs(\pos)$ similarly.


The following is a consequence of the synchronicity theorem (\cref{rem:sync}):
\begin{thm}\label{thm:faithful}
The $\Gamma_\bs$-actions on $\A_\bs$ and $\X_\bs$ are faithful. 
\end{thm}

\subsubsection*{Cluster DT transformation}
Fix an initial seed $\bi \in \bs$. A seed $\bi^\ast \in \bs$ is called \emph{terminal} if it satisfies $C^{\bi^\ast;\bi}=-\mathrm{Id}$. If such a seed exists, then there is a unique element $\tau^\bi \in \Gamma_\bs$ such that $\bi^\ast=(\tau^\bi)^{-1}(\bi)$. 
It is known that this element and its existence do not depend on the initial seed $\bi$ \cite[Theorem 3.6]{GS18}, and $\tau=\tau^\bi \in \Gamma_\bs$ is a central element \cite[Corollary 3.7]{GS18}. We call $\tau$ the \emph{cluster Donaldson--Thomas transformation} (\emph{cluster DT transformation} for short). We say that the mutation class $\bs$ (or the cluster modular group $\Gamma_\bs$) admits a cluster DT transformation if $\tau$ exists. 

Note that a representation sequence of $\tau$ is a reddening sequence in the sense of \cite{KellerDT}. A maximal green sequence \cite{KellerDT} is a reddening sequence, but the converse is not true in general.

\subsection{Cluster ensemble}
Let $\bs$ be a mutation class of seeds. The maps
\begin{align*}
    p_\bi: \bR_{>0}^I \to \bR_{>0}^I, \quad p_\bi^\ast X_k^\bi :=\prod_{i \in I} (A_j^\bi)^{\ve_{kj}^\bi}
\end{align*}
commute with cluster transformations \eqref{eq:A-transf}, \eqref{eq:X-transf}. In particular, they combine to define a $\Gamma_\bs$-equivariant real-analytic map
\begin{align*}
    p: \A_\bs(\pos) \to \X_\bs(\pos),
\end{align*}
which we call the \emph{ensemble map}.

The fibers of the ensemble map is described as follows \cite[Section 2.3]{FG09}. 
Fix a seed $\bi \in \bs$. 
For a vector $\beta=(\beta_j^\bi)\in \ker \ve^\bi \subset \bR^I$, Define a flow $\cF_{t\beta}: \A_\bs(\pos) \to \A_\bs(\pos)$ for $t \in \bR$ by the formula
\begin{align*}
    \cF_{t\beta}^\ast A_j^\bi :=e^{t\beta_j^\bi} A_j^\bi
\end{align*}
for $j \in I$. The flows $\cF_{t\beta}$ generate the fibers of $p$ \cite[Lemma 2.10]{FG09}. Namely, for any two points $g_0,g_1 \in \A_\bs(\pos)$ such that $p(g_0)=p(g_1)$, there exist $\beta \in \ker \ve^\bi$ such that $g_1=\cF_\beta(g_0)$. Indeed, we can take $\beta=\log\psi_\bi(g_1)-\log\psi_\bi(g_0)$. 

For any mutation $\bi'=\mu_k(\bi)$, we have $\ker \ve^{\bi'}=\ker \ve^\bi=:K^\vee$. 
For $\beta \in K^\vee$, we regard the components $\beta_j^\bi$ as coordinates of $\beta$ with respect to $\bi$. 
The mutation rule for $\beta \in K^\vee$ is
\begin{align*}
    \beta'_i = \begin{cases}
        \beta_i & \mbox{if $i \neq k$}, \\
        -\beta_k+ \sum_{j \in I} [\pm \ve_{kj}]_+ \beta_j & \mbox{if $i=k$}.
    \end{cases}
\end{align*}
Here, the sign in the second line does not matter. Via this mutation rule, we get a representation 
\begin{align}\label{eq:rep_fiber}
    \Gamma_\bs \to \mathrm{GL}(K^\vee).
\end{align}
The action map
\begin{align}\label{eq:fiber_action}
    K^\vee \times \A_\bs(\pos) \to \A_\bs(\pos)
\end{align}
is $\Gamma_\bs$-equivariant. 

Any cluster variable $A \in \CV_\bs$ is homogeneous with respect to this flow. Indeed, by the separation formula (\cref{thm:separation}), we get
\begin{align}\label{eq:CV_weight}
    \cF_{t\beta}^\ast A = e^{t \inprod{\mathbf{g}(A)^\bi}{\beta}} A
\end{align}
for $\beta=(\beta^\bi_j) \in \ker \ve^\bi$.


\section{Nielsen realization theorem}\label{sec:fixed_point_general}

\subsection{Nielsen realization theorem}
Let $\bs$ be a mutation class of seeds. We fix a linear ordering $\{1,\dots,n \} \cong I$ with $n=|I|$.

\begin{conv}
For any function $F: \A_\bs(\pos) \to \bR$ and a seed $\bi \in \bs$, let 
\begin{align*}
    F|_\bi:=F \circ \psi_\bi^{-1}: \bR_{>0}^n \to \bR \quad \mbox{and} \quad \lc{F}{\bi}:=F \circ \log\psi_\bi^{-1}: \bR^n \to \bR
\end{align*}
denote the coordinate expressions in the cluster and log-cluster chart, respectively. 
\end{conv}
The following is an immediate consequence of \cref{lem:log-poly_convex}:

\begin{lem}\label{lem:UCA_convex}
For any $F \in \UCA^+_\bs$ and $\bi \in \bs$, the function $\lc{\log F}{\bi}: \bR^n \to \bR$ is convex. 
\end{lem}

The next lemma is useful to investigate the strict convexity:

\begin{lem}\label{lem:CV_Span}
For any $A \in \CV_\bs$, we have $\mathrm{Slope}(A|_\bi) \subset (\ker \ve^\bi)^\bot$. 

Equivalently, $\mathrm{Slope}(A|_\bi)^\bot \supset \ker \ve^\bi$. 
\end{lem}

\begin{proof}
It follows from the separation formula \eqref{eq:separation_formula} that $\mathrm{Slope}(A|_\bi)$ is contained in the row span of $\ve^\bi$, which is $(\ker \ve^\bi)^\bot$.
\end{proof}

\begin{rem}\label{rem:CV_flat_direction}
From \cref{lem:CV_Span} and \cref{lem:strict_convex}, it follows that $\lc{\log A}{\bi}$ is linear along the direction $\ker \ve^\bi$ for any cluster variable $A \in \CV_\bs$. 
Note that \cref{lem:CV_Span} also holds for  any theta functions of \cite{GHKK}. 
\end{rem}

\begin{dfn}[filling sets]\label{def:filling_set}
A \emph{filling set} is a finite subset $\Lambda \subset \UCA^+_\bs$ satisfying the following conditions for some $\bi \in \bs$:
\begin{enumerate}
    \item \textbf{Balanced}: There exists a subset $S \subset \bigcup_{F \in \Lambda} \mathrm{Supp}(F|_\bi)$ such that $\mathrm{Conv}(S)$ is full-dimensional, and its interior contains $0$.
    \item \textbf{Slope-span}: We have $\mathrm{Slope}(F|_\bi) \subset (\ker \ve^\bi)^\bot$ for all $F \in \Lambda$, and
    \begin{align*}
        \mathrm{span}_\bR \bigcup_{F \in \Lambda}\mathrm{Slope}(F|_\bi) = (\ker \ve^\bi)^\bot. 
    \end{align*}
    for some $\bi \in \bs$.
\end{enumerate}
\end{dfn}
We remark that the second condition implies that $\lc{\log F}{\bi}$ is linear along the direction $\ker \ve^\bi$ (\cref{rem:CV_flat_direction}). 

Henceforth, we fix a filling set $\Lambda \subset \UCA^+_\bs$, and consider the function\footnote{Here, $L$ stands for ``Log''. $L_F$ should not be confused with the cluster length function \cite{GS19}, unless $F$ is a cluster monomial.}
\begin{align*}
    L_F:=\log F : \A_\bs(\pos) \to \bR
\end{align*}
for each $F \in \Lambda$, 
whose log-coordinate expressions are convex by \cref{lem:UCA_convex}. 
Let $(a_1^\bi,\dots,a^\bi_n)$ be the coordinates of the chart $\log \psi_\bi$. 

Let $G \subset \Gamma_\bs$ be any finite subgroup. 
Fixing any seed $\bi \in \bs$, we consider the function
\begin{align*}
    L_G:= \max \{ L_F \mid F \in \Lambda_G \}, \quad \Lambda_G:=\bigcup_{\phi \in G} \phi^\ast \Lambda,
\end{align*}
where $\phi^\ast$ stands for the pull-back action on $\UCA^+_\bs \subset C^\omega(\A_\bs(\pos))$. 
The log-coordinate expressions of $L_G$ is convex as the maximum of finitely many convex functions. 
Moreover, it is clearly $G$-invariant. 

\begin{lem}
For any filling set $\Lambda$ and finite subgroup $G \subset \Gamma_\bs$, 
the function $L_G: \A_\bs(\pos) \to \bR$ attains a unique minimum.
\end{lem}

\begin{proof}
Follows from \cref{lem:proper_convex}, using the balanced condition of $\Lambda$.
\end{proof}



Although the minimum (value) of $L_G$ is unique, it is possible that its minimizers (the points that realize the minimum) are not unique.
The following gives an analogue of \cite[Theorem 3]{Kerckhoff}: 

\begin{thm}\label{thm:minimum_unique_fiber}
For any filling set $\Lambda$ and finite subgroup $G \subset \Gamma_\bs$, 
the minimizer of the function $L_G$ is unique.
\end{thm}

\begin{proof}
Suppose $g_0,g_1 \in \A_\bs(\pos)$ are minimizers of $L_G$, and let $u_i:=\log\psi_\bi(g_i) \in \bR^I$ for $i=1,2$. Consider the direction vector $\beta:=u_1 - u_0 \in \bR^n$, and suppose $\beta \neq 0$.

\smallskip
\paragraph{\textbf{Step 1: Restriction to the fiber direction.}}
We claim $\beta \in K^\vee=\ker \ve^\bi$.
Consider the functions 
\begin{align*}
    L_{F,\beta}(t):=\lc{L_F}{\bi}(u_0+t\beta)\quad \mbox{and} \quad L_{G,\beta}(t):=\max\{L_F(t) \mid F \in \Lambda_G\}
\end{align*}
of $t \in \bR$. We also have the coordinate-free definition $L_{F,\beta}(t)=L_F(\cF_{t\beta}(g_0))$. 

First observe that $L_{G,\beta}(t)$ is strictly convex if $L_{F,\beta}(t)$ is strictly convex for some $F \in \Lambda_G$. 
If $\beta \notin \mathrm{Slope}(F|_\bi)^\bot$, then by \cref{lem:strict_convex}, $L_{F,\beta}(t)$ is strictly convex. It contradicts to the assumption that $g_0,g_1$ are both minimizers of $L_G$. Therefore we have
\begin{align*}
    \beta \in \bigcap_{F \in \Lambda} \mathrm{Slope}(F|_\bi)^\bot = \ker \ve^\bi 
\end{align*}
by the slope-span condition of $\Lambda$. 
Geometrically, it means $p(g_0)=p(g_1) \in \X_\bs(\pos)$.

\smallskip
\paragraph{\textbf{Step 2: Variation along the fiber direction}.}
We claim that $L_{G,\beta}(t)$ is strictly increasing. 
Since we have $\beta \in K^\vee$, let us consider the function 
\begin{align*}
    L_F^{\mathrm{fiber}}: K^\vee \to \bR, \quad L_F^{\mathrm{fiber}}(\beta):=L_F(\cF_\beta(g_0)),
\end{align*}
and let $L_G^\mathrm{fiber}:=\max\{L_F^\mathrm{fiber} \mid F \in \Lambda_G\}$. 

For each $F \in \Lambda_G$, we write
\begin{align*}
    L_F= \inprod{\alpha_F}{x} + \log \bigg(c_\alpha + \sum_{\gamma \in \mathrm{Supp}(F)}c_\gamma e^{\inprod{\gamma-\alpha_F}{x}}\bigg).
\end{align*}
Here, note that $\gamma-\alpha_F \in \mathrm{Slope}(F|_\bi) \subset (\ker \ve^\bi)^\bot$, and hence the second term is contant along the direction of $K^\vee$. Therefore $L_F^{\mathrm{fiber}}$ is an affine function with slope $\alpha_F$. 

Let $\Lambda_G(g_0):= \{ F \in \Lambda_G \mid L_F(g_0)=L_G(g_0)\}$ be the ``active'' subset at $g_0$. 
The subgradient set of $L_G^\mathrm{fiber}$ at $0$ is computed as \cite[Example D.3.4]{FundamentalConvex} 
\begin{align*}
    \partial L_G^\mathrm{fiber}(0) = \mathrm{Conv}(\{\alpha_F \mid F \in \Lambda_G(g_0)\}).
\end{align*}
Since $g_0$ is a minimizer, it follows from \cref{prop:min_subgrad} that $\partial L_G^\mathrm{fiber}(0)$ contains $0$. In particular, there exists $F \in \Lambda_G(g_0)$ such that $\inprod{\alpha_F}{\beta} >0$. It implies that $L_{F,\beta}(t)=L_F^\mathrm{fiber}(t\beta)$ increases linearly for all $t>0$. Hence $L_{G,\beta}(t)$ increases as well. Then we have $L_G(g_1) > L_G(g_0)$, which contradicts to the assumption that $g_0,g_1$ are both minimizers of $L_G$. Thus we conclude that $g_0=g_1$. 


\end{proof}


\begin{thm}[Nielsen realization]\label{thm:fixed_point_filling}
Assume that there exists a filling set $\Lambda \subset \UCA^+_\bs$. 
Then, any finite subgroup $G \subset \Gamma_\bs$ has a fixed point in $\A_\bs(\pos)$. In particular, it has a fixed point in $p(\A_\bs(\pos)) \subset \X_\bs(\pos)$.
\end{thm}

\begin{proof}
By the $G$-invariance of $L_G$, the group $G$ acts on the set of its minimizers. By \cref{thm:minimum_unique_fiber}, the unique minimizer $g_0$ gives a $G$-fixed point in $\A_\bs(\pos)$. By the $\Gamma_\bs$-equivariance of the ensemble map, its image $p(g_0) \in \X_\bs(\pos)$ is also a fixed point. 
\end{proof}

\begin{cor}\label{cor:fixed_point_DT}
Assume that $\Gamma_\bs$ admits a cluster DT transformation. Then, any finite subgroup $G \subset \Gamma_\bs$ has fixed poins in $\A_\bs(\pos)$ and $\X_\bs(\pos)$.
\end{cor}

\begin{proof}
Fix a seed $\bi \in \bs$, and consider the terminal seed $\bi^\ast$. We claim that 
\begin{align*}
    \Lambda:=\mathbf{A}^{\!\bi} \cup \mathbf{A}^{\!\bi_1}\cup \dots \mathbf{A}^{\!\bi_n}\cup \mathbf{A}^{\!\bi^\ast}
\end{align*}
gives a filling set. To verify the balanced condition, we first note $\mathbf{e}_k \in \mathrm{Supp}(A_k^\bi|_{\bi})$ for each $k=1,\dots,n$, where $\mathbf{e}_k$ denotes the unit vector. Moreover, by the separation formula (\cref{thm:separation}) and the tropical duality (\cref{thm:tropical_duality}), we have $-\mathbf{e}_k = \mathbf{g}(A_k^{\bi^\ast}) \in \mathrm{Supp}(A_k^{\bi^\ast}|_{\bi})$. 
Hence the balanced condition holds with $S:=\{\pm\mathbf{e}_k \mid k=1,\dots,n\}$. 

The inclusion $\mathrm{Slope}(A|_{\bi}) \subset (\ker \ve^\bi)^\bot$ holds for all $A \in \Lambda$ by \cref{lem:CV_Span}. 
By the formula \eqref{eq:A-transf}, we see that 
\begin{align*}
    \mathrm{Slope}(A'_k|_{\bi})= \mathrm{span}_\bR \{b_k\} ,
\end{align*}
where $b_k:=(\ve_{kj})_{j \in I}$ is the $k$-th row vector of the exchange matrix. Therefore all the row vectors of $\ve^\bi$ appears, and hence the slope-span condition holds. 

Then the assertion follows from \cref{thm:fixed_point_filling}.
\end{proof}

\begin{figure}[ht]
    \centering
    \includegraphics[width=0.5\linewidth]{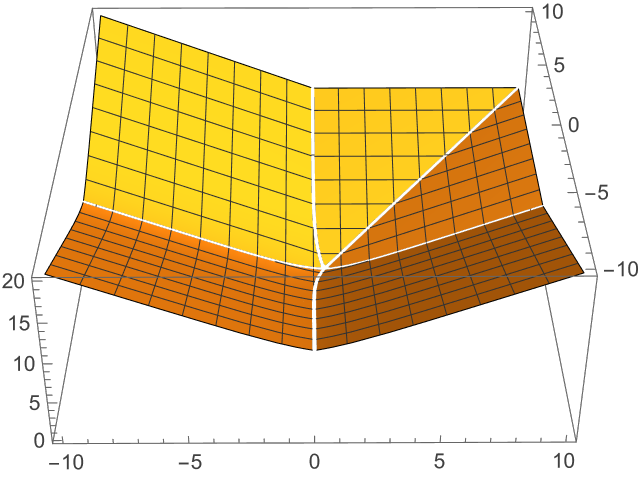}
    \caption{The function $L_G$ for type $A_2$.}
    \label{fig:A2_max_length}
\end{figure}

\begin{ex}[Type $A_2$]
Consider the mutation class $\bs$ containing the exchange matrix
\begin{align*}
    \ve^\bi=\begin{pmatrix}
    0 & 1 \\ -1 & 0
\end{pmatrix}.
\end{align*}
The set of all the cluster variables
\begin{align*}
    \Lambda:= \bigg\{ A_1,A_2, \frac{1+A_2}{A_1},\frac{1+A_1}{A_2},\frac{1+A_1+A_2}{A_1A_2} \bigg\}
\end{align*}
gives a filling set, which is invariant under $G=\Gamma_\bs=\bZ/5\bZ$.
\cref{fig:A2_max_length} shows the 3D plot of the function
\begin{align*}
    \lc{L_G}{\bi}(a_1,a_2) = \max\{ a_1,a_2, \log(1+e^{a_1})-a_2, \log(1+e^{a_2})-a_1, \log(1+e^{a_1}+e^{a_2})-a_1-a_2\}
\end{align*}
in the range $-10 \leq a_1,a_2 \leq 10$. It attains the unique minimum $\log\frac{1+\sqrt{5}}{2}$ at $a_1=a_2=\log\frac{1+\sqrt{5}}{2}$.
\end{ex}


\subsection{Nielsen realization theorem for finite mutation types}

\begin{prop}[Once-punctured surfaces]\label{prop:once-punctured}
Let $\bs=\bs_\Sigma$ be the mutation class associated with a once-punctured closed surface $\Sigma$ of genus $g \geq 1$ \cite{FST}. Then, any finite subgroup $G \subset \Gamma_\bs$ has fixed points in $\A_\bs(\pos)$ and $\X_\bs(\pos)$.
\end{prop}


\begin{proof}
This is a typical case where $\Gamma_\bs$ does not admit a cluster DT transformation. 
Consider the potential function $W \in \UCA^+_\bs$ associated to the unique puncture \cite{GS18}. We claim that
\begin{align*}
    \Lambda:=\mathbf{A}^{\!\bi} \cup \mathbf{A}^{\!\bi_1}\cup \dots \mathbf{A}^{\!\bi_n}\cup \{W\},
\end{align*}
is a filling set, where $\bi \in \bs$ denote the seed associated with an ideal triangulation $\tri$. The potential function is expressed as
\begin{align*}
    W|_{\bi} = \sum_{t} \bigg(\frac{A^{(t)}_1}{A^{(t)}_2 A^{(t)}_3}+\frac{A^{(t)}_2}{A^{(t)}_3 A^{(t)}_1}+\frac{A^{(t)}_3}{A^{(t)}_1 A^{(t)}_2} \bigg),
\end{align*}
where $t$ runs over all triangles of $\tri$, and $\{A_1^{(t)},A_2^{(t)},A_3^{(t)}\}$ denotes the cluster variables associated to the three sides of $T$. Since $\ker \ve^\bi=\mathrm{span}_\bR \{(1,\dots,1)\}$, 
the slope-span condition is verified similarly to the proof of \cref{cor:fixed_point_DT}. 
Since we have $\sum_{\alpha \in \mathrm{Supp}(W|_\bi)} \alpha = (-1,\dots,-1)$, the balanced condition holds with $S:=\{\mathbf{e}_j \mid j=1,\dots,n\} \cup \mathrm{Supp}(W|_\bi)$. Therefore the assertion holds by \cref{thm:fixed_point_filling}.
\end{proof}


\begin{thm}\label{thm:fixed_point_finite_mutation}
Let $\bs$ be any mutation class of finite mutation type, except for type $X_7$. Then, any finite subgroup $G \subset \Gamma_\bs$ has fixed points in $\A_\bs(\pos)$ and $\X_\bs(\pos)$. 
\end{thm}

\begin{proof}
Recall that the mutation classes of finite mutation type are classfied in \cite{FeST-1} for skew-symmetric case, and in \cite{FeST-2} in the skew-symmetrizable case.

In the skew-symmetric case, the cluster modular group $\Gamma_\bs$ admits a cluster DT transformation except for the following cases \cite[Theorem 7.4]{Mills}:
\begin{itemize}
    \item mutation class $\bs_\Sigma$ associated with a once-punctured closed surface $\Sigma$. 
    \item mutation class of type $X_7$. See \cref{fig:X_7} for the representing quiver.
\end{itemize}
Hence the skew-symmetric cases are covered by \cref{cor:fixed_point_DT} and \cref{prop:once-punctured}, except for $X_7$. 

In the skew-symmetrizable case, all the mutation classes of quivers listed in \cite[Theorem 5.13]{FeST-2} admit a cluster DT transformation. Indeed, \cite[Theorem 5.1 and Fig. 13]{KG24} shows all these quivers are mutation-equivalent to $T_{\mathbf{n},\mathbf{w}}$-quivers for some $\mathbf{n}$ and $\mathbf{w}$. Any $T_{\mathbf{n},\mathbf{w}}$-quiver admits a cluster DT tranformation \cite[Theorem 4.14]{KG24}. Again by \cref{cor:fixed_point_DT}, the assertion holds for skew-symmetrizable cases. 

\end{proof}

\begin{figure}[ht]
    \centering
\begin{tikzpicture}
\draw(0,0) circle(2pt) coordinate(A0);
\foreach \i in {1,2,3,4,5,6} {
    \draw(-\i*60-120:2) circle(2pt) coordinate(A\i);
    \node at (-\i*60-120:2.3) {\scriptsize $\i$};
}
\node at (0,0.5){\scriptsize $0$};
\foreach \i in {1,3,5} \qarrow{A0}{A\i};
\foreach \i in {2,4,6} \qarrow{A\i}{A0};
\draw[qarrow,transform canvas={yshift=-0.12em}] (A5) -- (A6);
\draw[qarrow,transform canvas={yshift=0.12em}] (A5) -- (A6);
\draw[qarrow,transform canvas={xshift=-0.12em,yshift=0.06em}] (A1) -- (A2);
\draw[qarrow,transform canvas={xshift=0.12em}] (A1) -- (A2);
\draw[qarrow,transform canvas={xshift=0.12em,yshift=0.06em}] (A3) -- (A4);
\draw[qarrow,transform canvas={xshift=-0.12em}] (A3) -- (A4);
\end{tikzpicture}
    \caption{A quiver in the mutation class of type $X_7$.}
    \label{fig:X_7}
\end{figure}
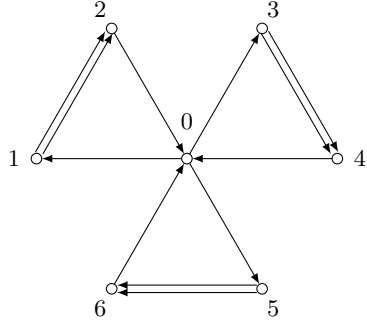

\begin{rem}\label{rem:X_7}
Only the remaining case of finite mutation type is $X_7$. We have only two quivers in this mutation class, one of which is shown in \cref{fig:X_7}. 
All the support vectors $\alpha$ of any cluster variables belong to the half-space defined by $\inprod{\alpha}{\beta} >0$, where $\beta:=(2,1,1,1,1,1,1)$ is the generator of $\ker \ve^\bi$.   

We have $K_j \in \UCA_\bs^+$ for $j=1,3,5$ given by
\begin{align*}
    K_j|_\bi:= \frac{A_j^2 + A_{j+1}^2+ A_0}{A_j A_{j+1}},
\end{align*}
whose support vectors belong to the same half-plane. To the author's knowledge, these are only elements in $\UCA_\bs \setminus \CA_\bs$ known at present (cf. \cite[Theorem 3.1]{Mills-UCA}). It implies that the balanced condition does not hold for any finite subset $\Lambda \subset \CV_\bs \cup \Gamma_\bs(\{K_1,K_3,K_5\})$, and hence the associated function $L_G$ is not bounded below. For example, we can directly see $\log F(\cF_{t\beta}(g_0)) \to -\infty$ as $t \to -\infty$
for all $F \in \Lambda$. 

Therefore, the strategy based on the Kerckhoff's argument seems not to work in this example. The author does not know if the Nielsen realization theorem holds for type $X_7$ or not.
\end{rem}

\subsection{Nielsen realization theorem for the cluster realization of Weyl groups}
In this subsection, we provide another strategy to prove the Nielsen realization theorem for cluster realizations of Weyl groups given in \cite{IIO}. 

Recall from \cite{IIO} that for any skew-symmetrizable Kac--Moody Lie algebra $\mathfrak{g}$ and an integer $m \geq 2$, we can construct a quiver $Q_m(\mathfrak{g})$ together with a group embedding 
\begin{align*}
    \phi_m: W(\mathfrak{g}) \to \Gamma_{\bs_m(\mathfrak{g})}
\end{align*}
of the Weyl group $W(\mathfrak{g})$, where $\bs_m(\mathfrak{g})$ is the mutation class containing $Q_m(\mathfrak{g})$. 

Recall the Coxeter presentation
\begin{align*}
    W(\mathfrak{g})= \langle r_s\ (s \in S) \mid (r_s r_u)^{m_{su}}=1\ (s,u \in S)\rangle,
\end{align*}
where $S$ is the vertex set of the Dynkin diagram of $\mathfrak{g}$, and $(m_{su})$ is a symmetric matrix satisfying $m_{ss}=1$ for all $s \in S$, other entries being related to the Cartan matrix $(C_{su})$ of $\mathfrak{g}$ as in the following table:

\begin{table}[ht]
\begin{tabular}{rccccc}
  $C_{su}C_{us}$ : & $0$ & $1$ & $2$ & $3$ & $\geq 4$  \\
  $m_{su}$ : & $2$ & $3$ & $4$ & $6$ & $\infty$. 
\end{tabular}
\label{tab:Coxeter}
\end{table}
The Weyl group $W(\mathfrak{g})$ is finite if and only if $\mathfrak{g}$ is of finite type, namely a finite-dimensional semisimple Lie algebra. 

The quiver $Q_m(\mathfrak{g})$ has the vertex set $I=\{v_i^s \mid i \in \bZ/m\bZ,\ s \in S\}$. It depends on a choice of \emph{Coxeter quiver} with vertex set $S$, having an exchange matrix $\ve=(\ve_{su})_{s,u \in S}$ such that $|\ve_{su}|=-C_{us}$ for $s \neq u$. Roughly speaking, $Q_m(\mathfrak{g})$ has an $m$-cycle for each $s \in S$, together with appropriate connection by arrows among them. 

By \cite[Theorem 3.10 (1)]{IIO}, the action of $\phi_m(r_s) \in \Gamma_{\bs_m(\mathfrak{g})}$ on the initial cluster variables is given as follows:
\begin{align}\label{eq:Weyl_action}
    \phi_m(r_s)^\ast A_j^u = \begin{cases}
        P_s \cdot A_j^u & \mbox{if $u=s$}, \\
        A_j^u & \mbox{if $u\neq s$}.
    \end{cases}
\end{align}
Here, $A_j^u:=A_{v_j^u}$ is the initial cluster variable associated with $v_j^u \in I$, and 
\begin{align}\label{eq:potential}
    P_s:= \sum_{j \in \bZ_m} \frac{1}{A_j^s A_{j+1}^s} \prod_{u \in S} (A_j^u)^{[-\ve_{su}]_+}(A_{j+1}^u)^{[\ve_{su}]_+}
\end{align}
is called the \emph{potential function} associated with $s \in S$. 

\begin{lem}\label{lem:action_peripheral}
The group $\phi_m(W(\mathfrak{g}))$ acts on $p(\A_{\bs_m(\mathfrak{g})})$ trivially, and acts on $K^\vee$ faithfully.
\end{lem}

\begin{proof}
The first statement is proved in \cite[Lemma 3.12]{IIO}. Then the second statement follows by \cref{thm:faithful} and the equivariance of \eqref{eq:fiber_action}.
\end{proof}
In particular, the Nielsen realization theorem holds obviously for $\X_{\bs_m(\mathfrak{g})}(\pos)$. Indeed, any point in $p(\A_{\bs_m(\mathfrak{g})})$ gives a fixed point of $\phi_m(W(\mathfrak{g}))$.


\begin{rem}
It is proved in \cite[Theorem 3.20 (1)]{IIO} that any reduced expression $w=r_{s_1}\dots r_{s_m} \in W(\mathfrak{g})$ gives rise to a green sequence. In particular, if $W(\mathfrak{g})$ is infinite, then the quiver $Q_m(\mathfrak{g})$ does not admit a maximal green sequence, since we have arbitrary long green sequences. Therefore, it is very likely that $\Gamma_{\bs_m(\mathfrak{g})}$ does not admit a cluster DT transformation.
\end{rem}

\begin{thm}\label{thm:Weyl_fixed_point}
Let $\mathfrak{g}$ be a skew-symmetrizable Kac--Moody Lie algebra, and $m \geq 2$ an integer. Then, for any finite subgroup $G \subset W(\mathfrak{g})$, $\phi_m(G)$ has a fixed point in $\A_{\bs_m(\mathfrak{g})}(\pos)$. 
\end{thm}

\begin{lem}\label{lem:Weyl_finite}
\cref{thm:Weyl_fixed_point} holds true for $\mathfrak{g}$ of finite type.
\end{lem}

\begin{proof}
In view of \eqref{eq:Weyl_action}, the assertion is equivalent to the existence of a point $g \in \A_{\bs_m(\mathfrak{g})}(\pos)$ satisfying $P_s(g)=1$ for all $s \in S$. We will work on the initial cluster where the action is described as in \eqref{eq:Weyl_action}, and omit the superscript $\bi$. 

For each $s \in S$, we have a vector $\beta_s=(\beta_j) \in K^\vee$ given by $\beta_{v_i^u}=\delta_{s,u}$ \cite[Theorem 3.13 (2)]{IIO}. From the homogeneity of the expression \eqref{eq:potential}, we see that the associated flow $\cF_{t\beta_s}$ rescales the potential functions as 
\begin{align*}
    P_u \mapsto e^{-tC_{us}} P_u.
\end{align*}
Since the Cartan matrix $C$ is invertible, for each $s \in S$, we may find a linear combination $\beta_s^\ast=\sum_{u \in S} a_u\beta_u \in K^\vee$ whose associated flow only rescales $P_s$ and keeps other potentials intact. Indeed, such a vector $\mathbf{a}=(a_u) \in \bR^S$ is given by $\mathbf{a}= - C^{-1}\cdot \mathbf{e}_s$, where $\mathbf{e}_s$ is the unit vector. By using the flows $\cF_{t\beta_s^\ast}$ for $s \in S$, we can change the values of potentials independently along the fiber direction. Thus we find a point $g \in \A_{\bs_m(\mathfrak{g})}(\pos)$ satisfying $P_s(g)=1$ for all $s \in S$. 
\end{proof}

\begin{rem}\label{rem:Weyl_section}
\begin{enumerate}
    \item The proof shows that the locus 
    \begin{align*}
        \{g \in \A_{\bs_m(\mathfrak{g})}(\pos) \mid P_s(g)=1,\ s \in S\}
    \end{align*}
    defines a $W(\mathfrak{g})$-equivariant section of the ensemble map $p: \A_{\bs_m(\mathfrak{g})}(\pos) \to \X_{\bs_m(\mathfrak{g})}(\pos)$.
    \item In the affine type $A$ case, there are no points $g \in \A_{\bs_m(\mathfrak{g})}(\pos)$ satisfying $P_s(g)=1$ for all $s\in S$. Indeed, it is easy to show that $\prod_{s \in S}P_s(g) > 1$ on $\A_{\bs_m(\mathfrak{g})}(\pos)$. It implies that there are no fixed points of the Weyl group $W(\mathfrak{g})$ (which is infinite) in this case. Since the Weyl group acts on the potentials as
    \begin{align*}
        \phi_m(r_u)^\ast P_s = P_s P_u^{-C_{us}}
    \end{align*}
    for $s,u \in S$, this fact also shows that the ensemble map does not admit any $W(\mathfrak{g})$-equivariant section. 
\end{enumerate}
\end{rem}

\begin{proof}[Proof of \cref{thm:Weyl_fixed_point}]
It is proved by Tits that any finite subgroup of a Coxeter group is contained in a spherical parabolic subgroup (see, for instance, \cite[Corollary D.2.9]{Davis}). Namely, there exists a subset $T \subset S$ and $w \in W(\mathfrak{g})$ such that $W_T$ is finite and $w G w^{-1} \subset W_T$. Here,
$W_T \subset W(\mathfrak{g})$ denotes the subgroup generated by $r_t$ for $t \in T$. By applying \cref{lem:Weyl_finite} to $W_T$, we find a fixed point $g \in \A_{\bs_m(\mathfrak{g})}(\pos)$ of $\phi_m(W_T)$. Then, $\phi_m(w)(g) \in \A_{\bs_m(\mathfrak{g})}(\pos)$ gives a fixed point of $\phi_m(G)$.
\end{proof}

\appendix

\section{Additional details for the fixed point theorem in \texorpdfstring{\cite{Ish19}}{[Ish19]}}\label{app:periodic}
The following result is stated in \cite[Proposition 2.3 (ii)$\Rightarrow$ (iii)]{Ish19}:

\begin{thm}\label{thm:fixed_point_Ish}
Let $\bs$ be any mutation class of seeds. 
Then, any finite-order element $\phi \in \Gamma_\bs$ has fixed points in $\A_\bs(\pos)$ and $\X_\bs(\pos)$. 
\end{thm}
We are going to fill a gap in its proof. The author is grateful to Oded Yacobi for pointing out an insufficiency of the argument in \cite{Ish19}. He also thanks Toshiyuki Akita for valuable comments.

Let $\cV$ denote either $\A$ or $\X$. The argument in \cite{Ish19} goes as follows. 
Take an element $\phi \in \Gamma_\bs$ of finite order, and consider its fixed point set
\begin{align*}
    \mathrm{Fix}_\cV(\phi):= \{ g \in \cV_\bs(\pos) \mid \phi(g)=g\}.
\end{align*}
We claim $\mathrm{Fix}_\cV(\phi) \neq \emptyset$. Suppose in contrary that $\mathrm{Fix}_\cV(\phi) = \emptyset$. 

Fixing a seed $\bi \in \bs$, consider the associated log-cluster chart $\log\psi_\bi: \cV_\bs(\pos) \xrightarrow{\sim} \bR^n$. The coordinate expression of $\phi$ gives a real-analytic diffeomorphism $\lc{\phi}{\bi}: \bR^n \to \bR^n$. 
Consider the one-point compactification $S^n = \bR^n \cup \{\infty\}$. Since $\lc{\phi}{\bi}$ is a proper map, it extends continuously to a homeomorphism $\lc{\overline{\phi}}{\bi}: S^n \to S^n$, whose only fixed point is $\infty$ by assumption. We then appeal to the Brown's theorem:

\begin{thm}[Brown {\cite[Theorem 5.1]{Brown}}]
Let $X$ be a paracompact space of finite cohomological dimension, and $s: X \to X$ a homeomorphism of finite order. If $H_\ast(\mathrm{Fix}(s^k);\bZ)$ is finitely generated for each $k \in \bZ$, then 
\begin{align*}
    \mathrm{Lef}(s):=\sum_{i \in \bZ} \mathrm{Tr}(s_\ast: H_i(X;\bZ) \to H_i(X;\bZ)) = \chi(\mathrm{Fix}(s)).
\end{align*}
Here, $\mathrm{Fix}(s^k):= \{x \in X \mid s^k(x)=x\}$.
\end{thm}
Applying Brown's theorem for $X=S^n$ and $s=\lc{\overline{\phi}}{\bi}$, we get a contradiction since the left-hand side is even, while the right-hand side, the Euler characteristic of a singleton, is $1$. Therefore $\mathrm{Fix}_\cV(\phi) \neq \emptyset$. 

In \cite{Ish19}, the verification of the finiteness of homology groups of $\mathrm{Fix}(s^k)$ was not made explicit. 
Note that there are examples of finite-order homeomorphisms on $\bR^n$ without fixed point \cite{Kister}. 
We state this homological finiteness in a general form, as follows:

\begin{prop}\label{prop:finite_cohomology}
For any $\phi \in \Gamma_\bs$, the fixed point set $\mathrm{Fix}_\cV(\lc{\overline{\phi}}{\bi}) \subset S^n$ of the extended action $\lc{\overline{\phi}}{\bi}: S^n \to S^n$ has finitely generated homology.
\end{prop}

\begin{rem}\label{rem:Hironaka}
\begin{enumerate}
    \item Here, we do not essentially need the tropical compactification $\overline{\cV_\bs(\pos)}=\cV_\bs(\pos) \cup \mathbb{P}\cV_\bs(\trop)$ considered in \cite{Ish19}. Nevertheless, it is useful to note that the sphere $S^n$ is obtained from $\overline{\cV_\bs(\pos)}$ by collapsing the boundary to a point, which provides us a coordinate-free description of the fixed point set in \cref{prop:finite_cohomology}.
    \item The fixed point set $\mathrm{Fix}_\cV(\phi) \subset \cV_\bs(\pos)$ viewed in any cluster chart is a real semialgebraic set, since it is defined by finitely many algebraic equations and inequalities coming from the positivity of coordinates. Hence, its homology groups are finitely generated by Hironaka's triangulation theorem \cite{Hironaka}. However, it remains to analysize its behavior ``at infinity''. 
\end{enumerate}
\end{rem}

To prove \cref{prop:finite_cohomology}, we use the framework of \emph{$o$-minimal geometry} that generalizes the real semialgebraic geometry. A concise account is found in \cite{Dries}. See \cite{Coste} for a brief introduction.

We will work over the ordered field $\bR=(\bR,+,\cdot,<)$. A \emph{structure} expanding $\bR$ is a collection $\cS=(\cS_n)_{n \in \mathbb{N}}$ of subsets $\cS_n \subset \mathcal{P}(\bR^n)$ satisfying the following conditions:
\begin{enumerate}
    \item[(S1)] All the algebraic subsets of $\bR^n$ are contained in $\cS_n$.
    \item[(S2)] For each $n \in \bN$, the set $\cS_n$ is closed under unions, intersections and complements. Namely, $\cS_n \subset \mathcal{P}(\bR^n)$ is a Boolean subalgebra.
    \item[(S3)] For each $m,n \in \bN$, $A \in \cS_m$ and $B \in \cS_n$, we have $A \times B \in \cS_{m+n}$.
    \item[(S4)] Let $p:\bR^{n+1} \to \bR^n$ be the projection to the first $n$ components. Then for any $A \in \cS_{n+1}$, we have $p(A) \in \cS_n$.
\end{enumerate}
It is said to be \emph{$o$-minimal} if it further satisfies the condition:
\begin{itemize}
    \item[(S5)] The elements of $\cS_1$ are precisely the finite unions of points and open intervals. 
\end{itemize}
The elements of $\cS_n$ are called \emph{definable subsets} of $\bR^n$ with respect to $\cS$. 
A map $f: A \to \bR^p$ defined on $A \subset \bR^n$ is called a \emph{definable map} if its graph is a definable subset of $\bR^{n+p}$. (It follows that $A$ is definable by the property (S4).) The image of a definable subset under a definable map is again definable. 

The $o$-minimality condition ensures that the definable subsets have ``tame'' topological properties. In particular, we have the following:

\begin{thm}[Triangulation theorem {\cite[Chapter 8, (1.7)]{Dries}}]\label{thm:triangulation}
Any definable subset $A \in \cS_n$ is definably homeomorphic to a polyhedron $|K|$ for some simplicial complex $K$ in $\bR^n$.
\end{thm}
Here, $K$ is a simplicial complex in the usual sense. In particular, it follows that $A$ has finitely generated homology.

It is a deep result of Wilkie \cite{Wilkie} that there exists an $o$-minimal structure $\bR_{\exp}$ generated by the exponential function.
It is the smallest $o$-minimal structure such that $\cS_2$ contains the graph of the exponential function. In particular, any functions $f: \bR^n \to \bR$ of the form \eqref{eq:log-poly} are definable with respect to $\bR_{\exp}$.

\begin{proof}[Proof of \cref{prop:finite_cohomology}]
Recall that the fixed point set $\mathrm{Fix}_\cV(\phi) \subset \cV_\bs(\pos)$ viewed in any cluster chart is a real semialgebraic set (\cref{rem:Hironaka} (2)). In particular, it is definable in any structure by (S1). When we view it in the log-cluster chart, it gives an $\bR_{\exp}$-definable subset of $\bR^n$. In order to obtain the one-point compactification, we embed $\iota: \bR^n \hookrightarrow S^n$ via the inverse of stereographic projection. Since $\iota$ is written by algebraic functions, the image of the fixed point set still gives a definable subset $F \subset S^n$. Finally, we obtain $\mathrm{Fix}_\cV(\lc{\overline{\phi}}{\bi}) \subset S^n$ as the union $F \cup \{\infty\}$, which is definable by (S2). Hence, $\mathrm{Fix}_\cV(\lc{\overline{\phi}}{\bi}) \subset S^n$ is homeomorphic to a polyhedron by \cref{thm:triangulation}, which has finitely generated homology. 
\end{proof}
Thus the proof of \cref{thm:fixed_point_Ish} is completed.

\end{document}